\documentclass[12pt]{article}
\usepackage{amssymb,amsmath,graphicx,amsthm,mathrsfs,bm}
\usepackage{color}
\newtheorem{Theorem}{Theorem}[section]
\newtheorem{Lemma}[Theorem]{Lemma}
\newtheorem{Proposition}[Theorem]{Proposition}

\newtheorem{Remark}[Theorem]{Remark}
\numberwithin{equation}{section}

\newcommand{\lc}
{\mathrel{\raise2pt\hbox{${\mathop<\limits_{\raise1pt\hbox
{\mbox{$\sim$}}}}$}}}

\newcommand{\gc}
{\mathrel{\raise2pt\hbox{${\mathop>\limits_{\raise1pt\hbox{\mbox{$\sim$}}}}$}}}

\newcommand{\ec}
{\mathrel{\raise2pt\hbox{${\mathop=\limits_{\raise1pt\hbox{\mbox{$\sim$}}}}$}}}

\begin{document}

\title{Minimal time impulse control of the heat equation}

\author{Yueliang Duan, \thanks{School of
Mathematics and Statistics,
Wuhan University, Wuhan 430072, China;
e-mail: duanyl@csu.edu.cn.}\quad
Lijuan Wang, \thanks{School of
Mathematics and Statistics, Computational Science Hubei Key Laboratory,
Wuhan University, Wuhan 430072, China;
e-mail: ljwang.math@whu.edu.cn. This work was supported by the
National Natural Science Foundation of China under grant 11771344.}
\quad Can Zhang
\thanks{School of Mathematics and Statistics,
Wuhan University, Wuhan 430072, China;
e-mail: canzhang@whu.edu.cn. This work was supported by the
National Natural Science Foundation of China under grant 11501424.}}

\date{}
\maketitle
\begin{abstract}
The paper is concerned with a kind of minimal time control problem for the heat equation with impulse controls.  The purpose of such a problem is to find an optimal
impulse control (among certain control constraint set) steering the solution of the heat equation from a given initial state to a given target set as soon as possible. We will first study the existence and uniqueness of optimal solution for this problem.  In the formulation of this problem, there are two parameters: one is the upper bound of the control constraint and the other one is the moment of impulse time. Then, we will establish the continuity of the minimal time function of this problem with respect to the above mentioned two parameters. Moreover, the convergence of the optimal
control is also discussed.
\end{abstract}

\medskip

\noindent\textbf{2010 Mathematics Subject Classifications.} 49K20, 49J20, 93C20

\medskip

\noindent\textbf{Keywords.} minimal time control, impulse control, bang-bang property

\section{Introduction}

Among the existing literature on the minimal time control problem for the evolution system, the control inputs are usually distributed in the whole time interval, i.e., they may affect the control system  at each instant of time
(see, for instance, \cite{Fattorini}, \cite{Lions},
 \cite{Phung}, \cite{Tucsnak}, \cite{Wang-Xu-Zhang}, \cite{Wang-Yan-0} and \cite{Can Zhang}). We also refer the readers to the minimal time sampled control problem of the heat equation (see, for instance, \cite{Ackermann},
\cite{Chen-Francis},  \cite{Franklin}, \cite{Ichikawa}, \cite{Landau} and \cite{Wang-Yang-Zhang}).
However, in many practical applications, it is much
more convenient to use impulse controls (see, for instance, \cite{Bensoussan},  \cite{Trelat}, \cite{Yang} and
\cite{Yong-Zhang}). In this paper, we will consider a kind of
minimal time  control problem for the heat equation with impulse controls.

Throughout the paper, $\Omega\subseteq \mathbb{R}^d$ $(d\geq 1)$  is a bounded domain with a $C^2$ smooth
boundary $\partial\Omega$; $\omega\subseteq \Omega$ is an open and nonempty subset with its
characteristic function $\chi_\omega$; $\lambda_1$ is the first eigenvalue of $-\Delta$ with
the homogeneous Dirichlet boundary on  $\partial\Omega$; $\{e^{\Delta t}\}_{t\geq 0}$ is the analytic semigroup (on $L^2(\Omega)$)
generated by $\Delta$ with its domain $H^2(\Omega)\cap H_0^1(\Omega)$; $B_r(0)$ denotes the closed ball
in $L^2(\Omega)$, centered at $0$ and of radius $r>0$.

Arbitrarily given an initial state $y_0\in L^2(\Omega)$, we formulate the following two
impulse controlled heat equations at the initial time zero and at a given time $\tau>0$, respectively:
\begin{equation}\label{impulse-1}
\left\{
\begin{array}{lll}
\partial_t y-\Delta y=0&\mbox{in}&\Omega\times (0,+\infty),\\
y=0&\mbox{on}&\partial\Omega\times (0,+\infty),\\
y(0)=y_0+\chi_\omega u&\mbox{in}&\Omega,
\end{array}\right.
\end{equation}
\begin{equation}\label{impulse-2}
\left\{
\begin{array}{lll}
\partial_t y-\Delta y=0&\mbox{in}&\Omega\times (0,+\infty),\\
y=0&\mbox{on}&\partial\Omega\times (0,+\infty),\\
y(0)=y_0,\;\;y(\tau)=y(\tau^-)+\chi_\omega u&\mbox{in}&\Omega,
\end{array}\right.
\end{equation}
where $u\in L^2(\Omega)$ is a control input activated on a subdomain $\omega$ of $\Omega$  and
$$
y(\tau^-)\triangleq\displaystyle{\lim_{t\rightarrow \tau^-}} y(t)\;\;\text{in}\;\;L^2(\Omega).
$$

In the sequel, we always write $y^0(\cdot; y_0, u)$ and $y^\tau(\cdot; y_0, u)$ for the solutions of the equations
(\ref{impulse-1}) and (\ref{impulse-2}), respectively.
It is well known that for each $T>0$, $y^0(\cdot; y_0, u)\in C([0,T];L^2(\Omega))$, and for each $T>\tau$,
$y^\tau(\cdot; y_0, u)\in C([0,\tau); L^2(\Omega))$ and $y^\tau(\cdot; y_0, u)\in C([\tau,T];L^2(\Omega))$.

For each $M>0$, we define a control constraint set $\mathcal{U}_M$ as follows:
\begin{equation*}
\mathcal{U}_M\triangleq \{u\in L^2(\Omega): \|u\|_{L^2(\Omega)}\leq M\}.
\end{equation*}
Given $\tau\geq 0$ and $M>0$, we now consider the minimal time control problem
\begin{equation}\label{impulse-2(1)}
(TP)_M^\tau:\;\;\;\;t^*(M,\tau)\triangleq \inf_{u\in \mathcal{U}_M} \{T\geq \tau: y^\tau(T; y_0, u)\in B_r(0)\}.
\end{equation}
Clearly, the optimal time depends on parameters $M$ and $\tau$. (It can be regarded as a function of these two parameters.)
For this problem $(TP)_M^\tau$, we say that $B_r(0)$ is the target set, and that $u\in \mathcal{U}_M$ is an admissible control if there exists $T\geq \tau$ so that $y^\tau(T; y_0, u)\in B_r(0)$;  we denote by $t^*(M,\tau)$ the optimal time if it exists, and by $u^*\in \mathcal{U}_M$ an optimal control if $y^\tau(t^*(M,\tau); y_0, u^*)\in B_r(0)$. We proved that $(TP)_M^\tau$ has at least one optimal control (see Lemma~\ref{lemma-1}).

We first note that if the relation
\begin{equation}\label{impulse-3}
\big\{e^{\tau \Delta} y_0+\chi_\omega u: u\in \mathcal{U}_M\big\}\cap B_r(0)=\emptyset
\end{equation}
does not hold, then $t^*(M,\tau)=\tau$. In such case, the problem $(TP)_M^\tau$ is trivial.
We then remark that if (\ref{impulse-3}) holds, then it also holds in a small neighborhood of $(M,\tau)$ (see
Lemma~\ref{2018-impulse-10}).

The main objectives of this paper are to address the following three questions under the assumption \eqref{impulse-3}:
\begin{description}
\item[($i$)] Does the problem $(TP)_M^\tau$ has a unique optimal control?

\item[($ii$)] Is the function $t^*(\cdot,\cdot)$ continuous at the point $(M,\tau)$?

\item[$(iii)$] Suppose that $(M_n, \tau_n)\rightarrow (M, \tau)$ as $n\rightarrow +\infty$. Does the optimal control
of the problem $(TP)_{M_n}^{\tau_n}$ converge to that of  the problem $(TP)_M^\tau$?
\end{description}

Our main results hereafter solve these three questions.

\begin{Theorem}\label{impulse-4}
Let $M>0$ and $\tau\geq 0$ satisfy  (\ref{impulse-3}). Then
the problem $(TP)_M^\tau$ has a unique optimal control $u^*$. Moreover,
$t^*(M,\tau)>\tau$ and $\|u^*\|_{L^2(\Omega)}=M$.
\end{Theorem}

\begin{Theorem}\label{impulse-5}
Let $M>0$ and $\tau\geq 0$ satisfy (\ref{impulse-3}).
Then the function $t^*(\cdot,\cdot)$ is continuous at the point $(M,\tau)$.
Furthermore, suppose that $(M_n,\tau_n)\rightarrow (M,\tau)$ as $n\rightarrow +\infty$, and that
$u_n^*$ and $u^*$ are the optimal controls for the problems $(TP)_{M_n}^{\tau_n}$
and $(TP)_M^\tau$, respectively. Then
\begin{equation*}\label{impulse-6}
u_n^*\rightarrow u^*\;\;\mbox{in}\;\;L^2(\Omega) \;\;\text{as}\;\;n\rightarrow +\infty.
\end{equation*}
\end{Theorem}

\begin{Remark}
Theorems \ref{impulse-4} and \ref{impulse-5} could be easily extended to a class of parabolic equations
$\partial_t y-\textrm{div}(A(\cdot)\nabla y)=0$ with either  homogenous Dirichlet or Neumann boundary conditions,
where $A(\cdot)$ is a symmetric matrix-valued function in $\Omega$ satisfying the uniform ellipticity and Lipschitz
conditions.
\end{Remark}

The minimal time impulse control problem of the heat equation  has not been touched upon till now. In the problem $(TP)_M^\tau$, the optimal time and
time optimal control are two of the most important quantities. In most papers concerning minimal time
control problems, people can  provide  necessary conditions for optimal controls, i.e.,
Pontryagin's maximum principle (see, for instance, \cite{Barbu}, \cite{Kunisch-1} and \cite{Li-Yong}). In some specific situations, people can also
give characteristics for the optimal time, as well as the optimal control for a minimal time control problem
(see, for instance, \cite{Kunisch-2}, \cite{Wang-Zuazua}
and \cite{Wang-Yan}).  We refer the reader to Remark \ref{impulse-7} for a characteristic for the optimal control
of the problem $(TP)_M^\tau$.

The dependence of the minimal time function  with respect to the
initial data has been analyzed in some earlier works (see, for instance, \cite{Bardi}, \cite{Carja} and \cite{Seidman}).
However, to the best of our knowledge, the continuity of the minimal time function $t^*(\cdot,\cdot)$ for the problem $(TP)_M^\tau$
(as a function of the upper bound $M$ of control constraint
and the impulse time $\tau$) is new. The difficulty is to obtain the continuity of $t^*(\cdot,\cdot)$ with respect to both two variables.
To over this difficulty, we need not only the continuity of $t^*(\cdot,\cdot)$ for each variable, but also
 the monotone property of $t^*(\cdot,\cdot)$
for the first variable.

Note however that obtaining the rates of continuity and convergence
is of interest, but  this issue is challenging and completely open.

The rest of the paper is organized as follows. Section \ref{s2} is devoted to the proof of Theorem~\ref{impulse-4},
while Theorem \ref{impulse-5} is proved in Section \ref{s3}.

\section{Proof of Theorem~\ref{impulse-4}}\label{s2}

We start with the existence of optimal controls to the problem $(TP)_M^\tau$ for any $M>0$ and any
$\tau\geq 0$.

\begin{Lemma}\label{lemma-1}
Let $M>0$ and $\tau\geq 0$. Then
the problem $(TP)_M^\tau$ has at lease one optimal control.
\end{Lemma}
\begin{proof}~We first show that
\begin{equation}\label{Proof-1}
0\;\;\mbox{is an admissible control to the problem}\;\;(TP)_M^\tau.
\end{equation}

Indeed, for each $T>0$, since $y^\tau(T; y_0, 0)=e^{\Delta T} y_0$,
by a standard energy estimate, we have that
\begin{equation}\label{Proof-2}
\|y^\tau(T; y_0, 0)\|_{L^2(\Omega)}\leq e^{-\lambda_1 T}\|y_0\|_{L^2(\Omega)}.
\end{equation}
From (\ref{Proof-2}) it follows that
\begin{equation}\label{Proof-2(1)}
\|y^\tau(T; y_0, 0)\|_{L^2(\Omega)}\leq r\;\;\mbox{when}\;\;
T=\frac{1}{\lambda_1}\mbox{ln}\frac{\|y_0\|_{L^2(\Omega)}}{r},
\end{equation}
which indicates (\ref{Proof-1}).\\

We next claim that
\begin{equation}\label{Proof-3}
(TP)_M^\tau\;\;\mbox{has at least one optimal control}.
\end{equation}

For this purpose, according to (\ref{Proof-1}) and (\ref{impulse-2(1)}), there exist  sequences
$\{T_n\}_{n\geq 1}\subseteq [\tau,+\infty)$  and
$\{u_n\}_{n\geq 1}\subseteq \mathcal{U}_M$ so that
\begin{equation}\label{Proof-4}
T_n\rightarrow t^*(M,\tau)\in [\tau,+\infty)\;\;\mbox{and}\;\;\{y^\tau(T_n; y_0, u_n)\}_{n\geq 1}\subseteq B_r(0).
\end{equation}
On one hand, since $\{u_n\}_{n\geq 1}\subseteq \mathcal{U}_M$, there exists a control
$\widetilde{u}\in \mathcal{U}_M$ and a subsequence of $\{u_n\}_{n\geq 1}$, still denoted in the same manner,
so that
\begin{equation}\label{Proof-5}
u_n\rightarrow \widetilde{u}\;\;\mbox{weakly in}\;\;L^2(\Omega).
\end{equation}
On the other hand, noting that
\begin{equation}\label{Proof-6}
y^\tau(T_n; y_0, u_n)=e^{\Delta T_n} y_0+e^{\Delta (T_n-\tau)}\chi_\omega u_n,
\end{equation}
by the first relation of (\ref{Proof-4}) and (\ref{Proof-5}), we can take the limit for $n\rightarrow +\infty$
in (\ref{Proof-6}) to obtain that
\begin{equation*}
y^\tau(T_n; y_0, u_n)\rightarrow y^\tau(t^*(M,\tau); y_0, \widetilde{u})\;\;
\mbox{weakly in}\;\;L^2(\Omega).
\end{equation*}
This, along with the second relation of (\ref{Proof-4}), implies that
\begin{equation*}
y^\tau(t^*(M,\tau); y_0, \widetilde{u})\in B_r(0).
\end{equation*}
Hence, $\widetilde{u}$ is an optimal control
to $(TP)_M^\tau$, i.e., (\ref{Proof-3}) follows.\\

In summary, we finish the proof of this lemma.
\end{proof}

We now turn to the proof of Theorem~\ref{impulse-4}.

\begin{proof}[Proof of Theorem~\ref{impulse-4}]
Let $M>0$ and $\tau\geq 0$ verify (\ref{impulse-3}).
According to Lemma~\ref{lemma-1} and (\ref{impulse-3}), we see that
$t^*(M,\tau)>\tau$ and $(TP)_M^\tau$ has an optimal control
$u^*$. For simplicity we set $t^*\triangleq t^*(M,\tau)$ and define
\begin{equation*}
\mathcal{A}_{t^*}\triangleq \big\{e^{t^*\Delta} y_0+e^{(t^*-\tau)\Delta}\chi_\omega u:
 u\in \mathcal{U}_M\big\}.
\end{equation*}
Obviously, $\mathcal{A}_{t^*}$ is a convex and closed subset of $L^2(\Omega)$.
The rest of the proof will be carried out by three steps as follows.\\

Step 1. We show that
\begin{equation}\label{Proof-8}
\mathcal{A}_{t^*}\cap B_r(0)=\big\{y^\tau(t^*; y_0, u^*)\big\}.
\end{equation}

Indeed, since $y^\tau(t^*; y_0, u^*)\in \mathcal{A}_{t^*}\cap B_r(0)$,
it suffices to show that $\mathcal{A}_{t^*}\cap B_r(0)$ has a unique element.
To seek a contradiction,
we would suppose that $\mathcal A_{t^*}\cap B_r(0)$ contains another
element (which is different from $y^\tau(t^*; y_0, u^*)$), denoted by
\begin{equation}\label{Proof-9}
y^\tau(t^*; y_0, \widehat{u})=
e^{t^*\Delta}y_0+e^{(t^*-\tau)\Delta}(\chi_\omega\widehat{u})\in B_r(0)
\end{equation}
with some $\widehat{u}\in\mathcal U_M$.
Set $v\triangleq (u^*+\widehat{u})/2$. It is clear that
\begin{equation}\label{Proof-10}
\|v\|_{L^2(\Omega)}\leq M\;\;\mbox{and}\;\;
y^\tau(t^*; y_0, v)=[y^\tau(t^*; y_0, u^*)+y^\tau(t^*; y_0, \widehat{u})]/2.
\end{equation}
Since $L^2(\Omega)$ is strictly convex, by (\ref{Proof-9}) and
the equality in (\ref{Proof-10}), we have that
\begin{equation*}
\|y^\tau(t^*; y_0, v)\|_{L^2(\Omega)}<r.
\end{equation*}
Noting that $t^*>\tau$,
we see that $y^\tau(\cdot; y_0, v)$ is continuous at the time $t^*$.
Hence, there exists $t_0\in (\tau,t^*)$ so that $\|y^\tau(t_0; y_0, v)\|_{L^2(\Omega)}\leq r$.
This, together with the inequality in (\ref{Proof-10}), leads to a contradiction with
the time optimality of $t^*$ for the problem $(TP)_M^{\tau}$. Thus, (\ref{Proof-8})
is verified.\\

Step 2. We claim that
\begin{equation}\label{Proof-10(1)}
\|u^*\|_{L^2(\Omega)}=M.
\end{equation}

For this purpose, since $t^*>\tau$, $y^\tau(\cdot; y_0, u^*)$ is continuous at the time $t^*$.
This implies that $$y^\tau(t^*; y_0, u^*)\in \mathcal{A}_{t^*}\cap \partial B_r(0),$$
which, combined with (\ref{Proof-8}), indicates that $\mathcal{A}_{t^*}\cap \text{int}B_r(0)=\emptyset$.
Here, $\text{int}B_r(0)$ denotes the interior of $B_r(0)$ in $L^2(\Omega)$.
Hence, according to the geometric version of the Hahn-Banach Theorem (see, for instance, \cite{Brezis}),
there exist $\varphi_0\in L^2(\Omega)$ with $\varphi_0\not=0$ and a constant $c$ so that
\begin{equation}\label{Proof-11}
\langle \varphi_0, z_1\rangle_{L^2(\Omega)}\geq c\geq
\langle \varphi_0, z_2\rangle_{L^2(\Omega)}\;\;\mbox{for all}\;\;
z_1\in B_r(0)\;\;\mbox{and}\;\;z_2\in \mathcal A_{t^*}.
\end{equation}
It follows from (\ref{Proof-11}) and (\ref{Proof-8}) that
\begin{equation}\label{Proof-12}
c=\langle \varphi_0, y^\tau(t^*; y_0, u^*)\rangle_{L^2(\Omega)}\;\;
\mbox{and}\;\;\langle \varphi_0, z_2-y^\tau(t^*; y_0, u^*)\rangle_{L^2(\Omega)}\leq 0\;\;
\mbox{for all}\;\;z_2\in \mathcal{A}_{t^*}.
\end{equation}
The inequality in (\ref{Proof-12}) yields that
\begin{equation*}
\langle u-u^*, \chi_\omega e^{(t^*-\tau)\Delta}\varphi_0\rangle_{L^2(\Omega)}\leq 0
\;\;\text{for all}\;\;u\in \mathcal{U}_M.
\end{equation*}
Hence,
\begin{equation}\label{Proof-13}
\max_{\|u\|_{L^2(\Omega)}\leq M}\langle u,\chi_\omega e^{(t^*-\tau)\Delta}\varphi_0\rangle_{L^2(\Omega)}
=\langle u^*,\chi_\omega e^{(t^*-\tau)\Delta}\varphi_0\rangle_{L^2(\Omega)}.
\end{equation}
Since $\varphi_0\neq 0$ and $t^*>\tau$, by the strong unique continuation property of the
heat equation (see, for instance, \cite{Lin}), we obtain that
$\chi_\omega e^{(t^*-\tau)\Delta}\varphi_0\neq 0$.
This, along with (\ref{Proof-13}), implies that
\begin{equation}\label{Proof-14}
u^*=M\frac{\chi_\omega e^{(t^*-\tau)\Delta}\varphi_0}{\|\chi_\omega e^{(t^*-\tau)\Delta}\varphi_0\|_{L^2(\Omega)}},
\end{equation}
which indicates (\ref{Proof-10(1)}).\\

Step 3. End of the proof.

Suppose that $v^*$ is also an optimal control to $(TP)_M^\tau$.
 It is clear that
$(u^*+v^*)/2$ is an optimal control to $(TP)_M^\tau$. According to (\ref{Proof-10(1)}),
\begin{equation*}
\|u^*\|_{L^2(\Omega)}=\|v^*\|_{L^2(\Omega)}=\|(u^*+v^*)/2\|_{L^2(\Omega)}=M.
\end{equation*}
These, together with the parallelogram rule, yield that
\begin{equation*}
\|u^*-v^*\|_{L^2(\Omega)}^2=2\big(\|u^*\|_{L^2(\Omega)}^2
+\|v^*\|^2_{L^2(\Omega)}\big)-\|u^*+v^*\|_{L^2(\Omega)}^2=0.
\end{equation*}
Hence, $u^*=v^*$.\\

In summary, we finish the proof of Theorem~\ref{impulse-4}.
\end{proof}

\begin{Remark}\label{impulse-7}
Moreover, we can give a characterization of the vector $\varphi_0$ in the proof of
Theorem~\ref{impulse-4}, and thus obtain a characterization of the optimal control to the problem $(TP)^\tau_M$.

In fact, we  see from (\ref{Proof-11}) and
the equality in (\ref{Proof-12}) that
\begin{equation}\label{impulse-8}
\langle \varphi_0, z-y^\tau(t^*; y_0, u^*)\rangle _{L^2(\Omega)}\geq 0\;\;\text{for all}\;\; z\in B_r(0).
\end{equation}
Since $y^\tau(t^*; y_0, u^*)\in B_r(0)$, it follows from (\ref{impulse-8}) that
\begin{equation*}
\langle \varphi_0, y^\tau(t^*; y_0, u^*)\rangle_{L^2(\Omega)}=\min_{z\in B_r(0)}
\langle \varphi_0, z\rangle_{L^2(\Omega)}=-r\|\varphi_0\|_{L^2(\Omega)}.
\end{equation*}
This yields that
\begin{equation*}
\varphi_0=c y^\tau(t^*; y_0, u^*)\;\;\text{for some constant}\;\;c<0,
\end{equation*}
which, combined with (\ref{Proof-14}), indicates that
\begin{equation*}
u^*=-M\frac{\chi_\omega e^{(t^*-\tau)\Delta}y^\tau(t^*; y_0, u^*)}
{\|\chi_\omega e^{(t^*-\tau)\Delta}y^\tau(t^*; y_0, u^*)\|_{L^2(\Omega)}}.
\end{equation*}
Thus, the unique optimal control $u^*$ to the problem $(TP)_M^{\tau}$
can be characterized by the following two-point boundary problem:
\begin{equation*}
\left\{
\begin{array}{lll}
\partial_t y-\Delta y=0&\text{in}&\Omega\times \big((0,\tau)\cup(\tau,t^*)\big),\\
y=0&\text{on}&\partial\Omega\times \big((0,\tau)\cup(\tau,t^*)\big),\\
y(0)=y_0,\;\;y(\tau)=y(\tau^-)+\chi_\omega u^*&\text{in}&\Omega,\\
\partial_t \varphi+\Delta \varphi=0&\text{in}&\Omega\times (0, t^*),\\
\varphi=0&\text{on}&\partial\Omega\times(0, t^*),\\
\varphi(t^*)=-y(t^*)&\text{in}&\Omega
\end{array}
\right.
\end{equation*}
with
\begin{equation*}
u^*=M\frac{\chi_\omega \varphi(\tau)}{\|\chi_\omega \varphi(\tau)\|_{L^2(\Omega)}}.
\end{equation*}
(Here, when $\tau=0$, $y(0)\triangleq y_0+\chi_\omega u^*$.)
\end{Remark}

\section{Proof of Theorem~\ref{impulse-5}}\label{s3}

In order to show the proof of Theorem~\ref{impulse-5}, we need some preparations.
\begin{Lemma}\label{2018-impulse-10} Let $M>0$ and $\tau\geq 0$ verify (\ref{impulse-3}).
Then there is a positive constant $\varepsilon_0\in (0, M)$
so that
\begin{equation}\label{proof-15}
\left\{e^{\widetilde{\tau}\Delta} y_0+\chi_\omega u:
\widetilde{\tau}\in [\max\{0, \tau-\varepsilon_0\},\tau+\varepsilon_0],
u\in \mathcal{U}_{M+\varepsilon_0}\right\}\cap B_r(0)=\emptyset.
\end{equation}
\end{Lemma}
\begin{proof}~By contradiction, there would exist two sequences $\{u_n\}_{n\geq 1}$ and $\{\tau_n\}_{n\geq 1}$ with
\begin{equation}\label{proof-16}
\|u_n\|_{L^2(\Omega)}\leq M+M/(2n)\;\;\mbox{and}\;\;
\tau_n\in [\max\{0, \tau-M/(2n)\},\tau+M/(2n)]
\end{equation}
so that
\begin{equation}\label{proof-17}
\|e^{\tau_n \Delta}y_0+\chi_\omega u_n\|_{L^2(\Omega)}\leq r\;\;\mbox{for all}\;\;n\geq 1.
\end{equation}
According to the first inequality of (\ref{proof-16}), there is a subsequence of $\{u_n\}_{n\geq 1}$, still
denoted in the same way, and a control $\widetilde{u}$ with $\|\widetilde{u}\|_{L^2(\Omega)}\leq M$, so that
\begin{equation*}
u_n\rightarrow \widetilde{u}\;\;\mbox{weakly in}\;\;L^2(\Omega).
\end{equation*}
This, along with the second inequality of (\ref{proof-16}), implies that
\begin{equation*}
e^{\tau_n \Delta} y_0+\chi_\omega u_n\rightarrow e^{\tau \Delta} y_0+\chi_\omega \widetilde{u}
\;\;\mbox{weakly in}\;\;L^2(\Omega),
\end{equation*}
which, combined with (\ref{proof-17}), indicates that
\begin{equation*}
\|e^{\tau \Delta} y_0+\chi_\omega \widetilde{u}\|_{L^2(\Omega)}\leq r.
\end{equation*}
This leads to a contradiction with (\ref{impulse-3}).
(Here, we used the fact that $\widetilde{u}\in \mathcal{U}_M$.)\\

Hence, we finish the proof of this lemma.
\end{proof}

The second result is concerned with the continuity of the function $t^*(M,\cdot)$, where
$M>0$ is fixed.
\begin{Proposition}\label{impulse-9} Let $M>0$ and $\tau\geq 0$ satisfy (\ref{impulse-3}). Then the function $t^*(M,\cdot)$ is continuous at the time $\tau$.
\end{Proposition}
\begin{proof}~We arbitrarily fix $\{\tau_n\}_{n\geq 1}\subseteq (0,+\infty)$ with $\tau_n\rightarrow \tau$.
By Lemma~\ref{2018-impulse-10}, we can assume that
\begin{equation}\label{proof-1}
\{e^{\tau_n \Delta} y_0+\chi_\omega u: u\in \mathcal{U}_M\}\cap B_r(0)=\emptyset\;\;\mbox{for each}
\;\;n\geq 1.
\end{equation}
This, along with (\ref{impulse-3}) and Theorem~\ref{impulse-4}, implies that
\begin{equation}\label{proof-2}
t^*(M,\tau)>\tau\;\;\mbox{and}\;\;t^*(M,\tau_n)>\tau_n\;\;\mbox{for each}\;\;n\geq 1.
\end{equation}
Moreover, by (\ref{Proof-2(1)}) in the proof of Lemma~\ref{lemma-1}, we see that
\begin{equation*}
0\leq t^*(M,\tau_n)\leq \frac{1}{\lambda_1}\mbox{ln}\frac{\|y_0\|_{L^2(\Omega)}}{r}
\;\;\mbox{for each}\;\;n\geq 1.
\end{equation*}
We aim to show that
\begin{equation}\label{proof-3}
\lim_{n\rightarrow +\infty} t^*(M,\tau_n)=t^*(M,\tau).
\end{equation}
The rest of the proof will be split into the following three steps.\\

Step 1. We claim that
\begin{equation}\label{proof-4}
t^*(M,\tau)\leq \liminf_{n\rightarrow +\infty} t^*(M,\tau_n)\triangleq \widetilde{t}.
\end{equation}

Without loss of generality,  we suppose that there is a subsequence of $\{n\}_{n\geq 1}$,
still denoted in the same way, so that
\begin{equation}\label{proof-5}
\lim_{n\rightarrow +\infty} t^*(M,\tau_n)=\widetilde{t}.
\end{equation}
This, together with the second inequality of (\ref{proof-2}), implies that
$\widetilde{t}\geq \tau$. For each $n\geq 1$, according to (\ref{proof-1}) and Theorem~\ref{impulse-4},
$(TP)_M^{\tau_n}$ has a unique optimal control $u_n^*\in \mathcal{U}_M$. Then
\begin{equation}\label{proof-6}
\begin{array}{lll}
&&y^{\tau_n}(t^*(M,\tau_n); y_0, u_n^*)
=e^{t^*(M,\tau_n)\Delta} y_0
+e^{(t^*(M,\tau_n)-\tau_n)\Delta} \chi_\omega u_n^*\in B_r(0),
\end{array}
\end{equation}
and there exists a subsequence of $\{n\}_{n\geq 1}$, still denoted by itself, and
a control $\widetilde{u}\in \mathcal{U}_M$, so that
\begin{equation}\label{proof-7}
u_n^*\rightarrow \widetilde{u}\;\;\mbox{weakly in}\;\;L^2(\Omega).
\end{equation}
Passing to the limit as $n\rightarrow +\infty$ in (\ref{proof-6}),
by (\ref{proof-5}) and (\ref{proof-7}), we see that
\begin{equation}\label{proof-7(1)}
y^{\tau_n}(t^*(M,\tau_n); y_0, u_n^*)\rightarrow y^\tau(\widetilde{t}; y_0, \widetilde{u})\;\;
\mbox{weakly in}\;\;L^2(\Omega)\;\;
\mbox{and}\;\;y^\tau(\widetilde{t}; y_0, \widetilde{u})\in B_r(0).
\end{equation}
Hence, $\widetilde{u}$ is an admissible control to the problem $(TP)_M^\tau$ and
$t^*(M,\tau)\leq \widetilde{t}$. Then (\ref{proof-4}) follows.\\

Step 2. We show that
\begin{equation}\label{proof-8}
t^*(M,\tau)\geq \limsup_{n\rightarrow +\infty} t^*(M,\tau_n)\triangleq \widehat{t}.
\end{equation}

To seek a contradiction, we would suppose that $\widehat{t}>t^*(M,\tau)$. Without loss of
generality, we assume that there exists a subsequence of $\{n\}_{n\geq 1}$, still
denoted in the same manner, so that
\begin{equation}\label{proof-9}
\lim_{n\rightarrow +\infty} t^*(M,\tau_n)=\widehat{t}.
\end{equation}
We now choose a positive constant $\delta\in \big(0,(\widehat{t}-t^*(M,\tau))/2\big)$.
According to (\ref{proof-9}), there is a positive integer $n_1(\delta)$ so that
\begin{equation}\label{proof-10}
t^*(M,\tau_n)>t^*(M,\tau)+\delta\;\;\mbox{for all}\;\;n\geq n_1(\delta).
\end{equation}
Let $u^*\in \mathcal{U}_M$ be the unique optimal control to the problem $(TP)_M^\tau$
(see Theorem~\ref{impulse-4}). It is clear that
\begin{equation}\label{proof-11}
\begin{array}{lll}
&&y^\tau(t^*(M,\tau); y_0, u^*)\\
&=&e^{t^*(M,\tau)\Delta} y_0+e^{(t^*(M,\tau)-\tau)\Delta}(\chi_\omega u^*)\in B_r(0).
\end{array}
\end{equation}
On one hand, since $\tau_n\rightarrow \tau<t^*(M,\tau)$ (see the first inequality of (\ref{proof-2})),
there is a positive integer $n_2(\delta)\geq n_1(\delta)$ so that for all $n\geq n_2(\delta)$,
\begin{equation*}
t^*(M,\tau)>\tau_n
\end{equation*}
and
\begin{equation}\label{proof-13}
\|e^{(t^*(M,\tau)-\tau_n)\Delta}(\chi_\omega u^*)-e^{(t^*(M,\tau)-\tau)\Delta}(\chi_\omega u^*)\|_{L^2(\Omega)}
\leq r(e^{\delta \lambda_1}-1).
\end{equation}
On the other hand, since
\begin{eqnarray*}
y^{\tau_n}(t^*(M,\tau); y_0, u^*)&=&e^{t^*(M,\tau)\Delta}y_0+e^{(t^*(M,\tau)-\tau_n)\Delta}(\chi_\omega u^*)\\
&=&\left[e^{t^*(M,\tau)\Delta}y_0+e^{(t^*(M,\tau)-\tau)\Delta}(\chi_\omega u^*)\right]\\
&&+\left[e^{(t^*(M,\tau)-\tau_n)\Delta} (\chi_\omega u^*)-e^{(t^*(M,\tau)-\tau)\Delta}(\chi_\omega u^*)\right],
\end{eqnarray*}
by (\ref{proof-11}) and (\ref{proof-13}), we have that
\begin{equation*}
\|y^{\tau_n}(t^*(M,\tau); y_0, u^*)\|_{L^2(\Omega)}\leq r e^{\delta \lambda_1}\;\;
\mbox{for all}\;\;n\geq n_2(\delta).
\end{equation*}
This, along with the decay of the energy for solutions to the heat equation, implies that
\begin{eqnarray*}
\|y^{\tau_n}(t^*(M,\tau)+\delta; y_0, u^*)\|_{L^2(\Omega)}&=&
\|e^{\delta\Delta} y^{\tau_n}(t^*(M,\tau); y_0, u^*)\|_{L^2(\Omega)}\\
&\leq& e^{-\delta \lambda_1} \|y^{\tau_n}(t^*(M,\tau); y_0, u^*)\|_{L^2(\Omega)}\\
&\leq&r\;\;\mbox{for all}\;\;n\geq n_2(\delta).
\end{eqnarray*}
It follows from the latter and the time optimality of $t^*(M,\tau_n)$
for the problem $(TP)_M^{\tau_n}$ that
\begin{equation*}
t^*(M,\tau_n)\leq t^*(M,\tau)+\delta\;\;\mbox{for all}\;\;n\geq n_2(\delta)\geq n_1(\delta),
\end{equation*}
which leads to a contradiction with (\ref{proof-10}). Hence, (\ref{proof-8}) holds.\\

Step 3. End of the proof.

According to (\ref{proof-4}) and (\ref{proof-8}), we arrive at (\ref{proof-3}).\\

Hence, we finish the proof of Proposition~\ref{impulse-9}.
\end{proof}

The next result is concerned with the monotonicity and the continuity of the function $t^*(\cdot,\tau)$,
where $\tau$ is fixed.
\begin{Proposition}\label{impulse-10} Let $M>0$ and $\tau\geq 0$ verify
(\ref{impulse-3}). Then the function $t^*(\cdot,\tau)$ is
strictly monotone decreasing near $M$ and is continuous at $M$.
\end{Proposition}
\begin{proof}~Let $\varepsilon_0$ be the constant in Lemma~\ref{2018-impulse-10}. The proof
will be split into three steps as follows.\\

Step 1. We show that for each $\widetilde{\tau}\in [\max\{0,\tau-\varepsilon_0\},\tau+\varepsilon_0]$,
\begin{equation}\label{proof-18}
t^*(M_2,\widetilde{\tau})<t^*(M_1,\widetilde{\tau})\;\;
\mbox{for all}\;\;M-\varepsilon_0<M_1<M_2< M+\varepsilon_0.
\end{equation}

Indeed, by Lemma~\ref{2018-impulse-10} and Theorem~\ref{impulse-4}, for each $i=1, 2$, we have
that $t^*(M_i,\widetilde{\tau})>\widetilde{\tau}$, and $(TP)_M^{\tau_i}$ has a unique optimal control
$u_i^*\in \mathcal{U}_{M_i}$ so that
\begin{equation}\label{proof-19}
y^\tau(t^*(M_i,\widetilde{\tau}); y_0, u_i^*)\in B_r(0)\;\;\mbox{and}\;\;\|u_i^*\|_{L^2(\Omega)}=M_i.
\end{equation}
Since $M_1<M_2$, by the time optimality of $t^*(M_2,\widetilde{\tau})$ and (\ref{proof-19}), we see
that $t^*(M_2,\widetilde{\tau})\leq t^*(M_1,\widetilde{\tau})$. If
$t^*(M_2,\widetilde{\tau})=t^*(M_1,\widetilde{\tau})$, noting that $M_1<M_2$,
 by (\ref{proof-19}) again, we obtain that $u_1^*$ is also an optimal control to
the problem $(TP)_{M_2}^{\widetilde{\tau}}$. Since $u_2^*$ is the unique optimal control to
the problem $(TP)_{M_2}^{\widetilde{\tau}}$, we have that $u_1^*=u_2^*$ and
\begin{equation*}
M_1=\|u_1^*\|_{L^2(\Omega)}=\|u_2^*\|_{L^2(\Omega)}=M_2,
\end{equation*}
which leads to a contradiction. Hence, (\ref{proof-18}) holds.\\

Step 2. We prove that $t^*(\cdot,\tau)$ is right continuous at $M$.

For this purpose, let $\{M_n\}_{n\geq 1}\subseteq (0,+\infty)$ be a strictly monotone
decreasing sequence so that $M_n\rightarrow M$. By Lemma~\ref{2018-impulse-10}, we can assume that
\begin{equation}\label{proof-20}
\{e^{\tau \Delta} y_0+\chi_\omega u: u\in \mathcal{U}_{M_n}\}\cap B_r(0)=\emptyset\;\;\mbox{for each}\;\;n\geq 1.
\end{equation}
This, along with Theorem~\ref{impulse-4} and Step 1, yields that
\begin{equation*}
\tau<t^*(M_1,\tau)<t^*(M_2,\tau)<\cdots<t^*(M_n,\tau)<\cdots<t^*(M,\tau),
\end{equation*}
which indicates that
\begin{equation}\label{proof-21}
\tau<\widetilde{t}=\lim_{n\rightarrow +\infty} t^*(M_n,\tau)\leq t^*(M,\tau).
\end{equation}
According to (\ref{proof-20}) and Theorem~\ref{impulse-4}, we see that $(TP)_{M_n}^\tau$ has a
unique optimal control $u_n^*$. Moreover,
 $$\|u_n^*\|_{L^2(\Omega)}=M_n\leq M_1\;\;\text{for all}\;\; n\geq 1.$$
Hence, there exists a subsequence of $\{u_n^*\}_{n\geq 1}$, still denoted in the same way,
and a control $\widetilde{u}\in L^2(\Omega)$ so that
\begin{equation}\label{proof-22}
u_n^*\rightarrow \widetilde{u}\;\;\mbox{weakly in}\;\;L^2(\Omega)
\end{equation}
and
\begin{equation}\label{proof-23}
\|\widetilde{u}\|_{L^2(\Omega)}\leq \liminf_{n\rightarrow +\infty} \|u_n^*\|_{L^2(\Omega)}=M.
\end{equation}
Noting that
\begin{eqnarray*}
&&y^\tau(t^*(M_n,\tau); y_0, u_n^*)\\
&=&e^{t^*(M_n,\tau)\Delta} y_0+e^{(t^*(M_n,\tau)-\tau)\Delta}(\chi_\omega u_n^*)\in B_r(0),
\end{eqnarray*}
by (\ref{proof-21}) and (\ref{proof-22}),
we have that
\begin{equation}\label{proof-24}
y^{\tau}(t^*(M_n,\tau); y_0, u_n^*)\rightarrow y^\tau(\widetilde{t}; y_0, \widetilde{u})\;\;
\mbox{weakly in}\;\;L^2(\Omega)\;\;
\mbox{and}\;\;y^\tau(\widetilde{t}; y_0, \widetilde{u})\in B_r(0).
\end{equation}
It follows from the second relation of (\ref{proof-24}) and  (\ref{proof-23}) that
$\widetilde{u}$ is an admissible control to the problem $(TP)_M^\tau$. Hence, $t^*(M,\tau)\leq \widetilde{t}$.
This, along with (\ref{proof-21}), implies that
\begin{equation*}
\lim_{n\rightarrow +\infty} t^*(M_n,\tau)=t^*(M,\tau).
\end{equation*}
Thus, the desired result holds.\\

Step 3. We claim that $t^*(\cdot,\tau)$ is left continuous at $M$.

To this end, let $\{M_n\}_{n\geq 1}\subseteq (0,+\infty)$ be a strictly monotone increasing sequence so that
$M_n\rightarrow M$. By (\ref{impulse-3}), Theorem~\ref{impulse-4} and Step 1, we observe that
\begin{equation*}
t^*(M_1,\tau)>t^*(M_2,\tau)>\cdots>t^*(M_n,\tau)>\cdots>t^*(M,\tau)>\tau.
\end{equation*}
This implies that
\begin{equation}\label{proof-25}
\widehat{t}=\lim_{n\rightarrow +\infty} t^*(M_n,\tau)\geq t^*(M,\tau)>\tau.
\end{equation}
It suffices to show that
\begin{equation}\label{proof-26}
\widehat{t}=t^*(M,\tau).
\end{equation}
To seek a contradiction, we would suppose that
\begin{equation*}
\widehat{t}>t^*(M,\tau).
\end{equation*}
Let $\delta\in \left(0,\min\{(\widehat{t}-t^*(M,\tau))/2, \text{ln}2/\lambda_1\}\right)$.
According to (\ref{proof-25}), there is a positive integer $n_1(\delta)$ so that
\begin{equation}\label{proof-27}
t^*(M_n,\tau)>t^*(M,\tau)+\delta\;\;\mbox{for all}\;\;n\geq n_1(\delta).
\end{equation}
Set $\sigma_n\triangleq M_n/M$ for all $n\geq 1$. Obviously, there is a positive integer
$n_2(\delta)\geq n_1(\delta)$ so that
\begin{equation}\label{proof-28}
1\geq \sigma_n\geq 1-r(e^{\delta \lambda_1}-1)/\|y_0\|_{L^2(\Omega)}\;\;\mbox{for all}\;\;n\geq n_2(\delta).
\end{equation}
By (\ref{impulse-3}) and Theorem~\ref{impulse-4}, $(TP)_M^\tau$ has a unique optimal control
$u^*$ with $\|u^*\|_{L^2(\Omega)}=M$ so that
\begin{equation}\label{proof-29}
\begin{array}{lll}
&&y^\tau(t^*(M,\tau); y_0, u^*)\\
&=&e^{t^*(M,\tau)\Delta} y_0+e^{(t^*(M,\tau)-\tau)\Delta}(\chi_\omega u^*)\in B_r(0).
\end{array}
\end{equation}
Denote $u_n\triangleq \sigma_n u^*$ for all $n\geq n_2(\delta)$. It is obvious that
\begin{equation}\label{proof-30}
\|u_n\|_{L^2(\Omega)}=M_n\;\;\mbox{for all}\;\;n\geq n_2(\delta).
\end{equation}
Since $t^*(M,\tau)>\tau$ and
\begin{eqnarray*}
y^\tau(t^*(M,\tau); y_0, u_n)&=&e^{t^*(M,\tau)\Delta}y_0+e^{(t^*(M,\tau)-\tau)\Delta}(\chi_\omega u_n)\\
&=&\sigma_n\left(e^{t^*(M,\tau)\Delta}y_0+e^{(t^*(M,\tau)-\tau)\Delta}(\chi_\omega u^*)\right)
+(1-\sigma_n)e^{t^*(M,\tau)\Delta}y_0\\
&=&\sigma_n y^\tau(t^*(M,\tau); y_0, u^*)+(1-\sigma_n)e^{t^*(M,\tau)\Delta}y_0,
\end{eqnarray*}
by the decay of the energy for solutions to the heat equation,
we obtain that
\begin{eqnarray*}
&&\|y^\tau(t^*(M,\tau)+\delta; y_0, u_n)\|_{L^2(\Omega)}\\
&\leq&e^{-\lambda_1\delta}\|y^\tau(t^*(M,\tau); y_0, u_n)\|_{L^2(\Omega)}\\
&\leq&e^{-\lambda_1\delta}(\sigma_n \|y^\tau(t^*(M,\tau); y_0, u^*)\|_{L^2(\Omega)}
+(1-\sigma_n)e^{-\lambda_1 t^*(M,\tau)}\|y_0\|_{L^2(\Omega)}).
\end{eqnarray*}
This, along with (\ref{proof-29}) and (\ref{proof-28}), implies that
\begin{equation}\label{proof-31}
\begin{array}{lll}
&&\|y^\tau(t^*(M,\tau)+\delta; y_0, u_n)\|_{L^2(\Omega)}\\
&\leq&e^{-\lambda_1\delta}[\sigma_n r+(1-\sigma_n)\|y_0\|_{L^2(\Omega)}]\leq r\;\;\mbox{for all}\;\;n\geq n_2(\delta).
\end{array}
\end{equation}
It follows from (\ref{proof-31}) and (\ref{proof-30}) that $u_n$ is an admissible control to
the problem $(TP)_{M_n}^\tau$ for each $n\geq n_2(\delta)$. Hence, by the time optimality of $t^*(M_n,\tau)$,
we get that
\begin{equation*}
t^*(M_n,\tau)\leq t^*(M,\tau)+\delta\;\;\mbox{for all}\;\;n\geq n_2(\delta)\geq n_1(\delta).
\end{equation*}
This leads to a contradiction with (\ref{proof-27}). Thus, the desired claim is true.\\

In summary, by Steps 1-3, we finish the proof of Proposition~\ref{impulse-10}.
\end{proof}

Finally, we turn to the proof of Theorem~\ref{impulse-5}.

\begin{proof}[Proof of Theorem~\ref{impulse-5}]
Let $M>0$ and $\tau\geq 0$ verify (\ref{impulse-3}). Let
$\varepsilon_0\in (0, M)$ be defined as in Lemma~\ref{2018-impulse-10}. According to (\ref{impulse-3})
and Proposition~\ref{impulse-10}, $t^*(\cdot,\tau)$ is continuous at $M$. Then for any
$\varepsilon>0$, there is a constant $\delta_1(\varepsilon)\in (0,\varepsilon_0/2)$ so that
\begin{equation}\label{proof-32}
|t^*(\widetilde{M},\tau)-t^*(M,\tau)|\leq \varepsilon/2
\;\;\mbox{when}\;\;|\widetilde{M}-M|\leq \delta_1(\varepsilon).
\end{equation}
For the same $\varepsilon$ above, by Lemma~\ref{2018-impulse-10} and Proposition~\ref{impulse-9}
(where $M$ is chosen as $M+\delta_1(\varepsilon)$ and $M-\delta_1(\varepsilon)$, respectively), there is
a constant $\delta_2(\varepsilon)\in (0,\delta_1(\varepsilon))$ so that for all $\widetilde{\tau}\in
[\max\{0,\tau-\delta_2(\varepsilon)\},\tau+\delta_2(\varepsilon)]$,
\begin{equation}\label{proof-33}
\begin{array}{l}
|t^*(M+\delta_1(\varepsilon),\widetilde{\tau})-t^*(M+\delta_1(\varepsilon),\tau)|\\
+|t^*(M-\delta_1(\varepsilon),\widetilde{\tau})-t^*(M-\delta_1(\varepsilon),\tau)|
<\varepsilon/2.
\end{array}
\end{equation}
Hence, it follows from (\ref{proof-32}), (\ref{proof-33}) and (\ref{proof-18}) that as
$\big|\widetilde{M}-M\big|\leq \delta_1(\varepsilon)$ and $\widetilde{\tau}\in
[\max\{\tau-\delta_2(\varepsilon),0\},\tau+\delta_2(\varepsilon)]$,
\begin{eqnarray*}
&&\big|t^*\big(\widetilde{M},\widetilde{\tau}\big)-t^*(M,\tau)\big|\\
&\leq&\max\big\{|t^*(M+\delta_1(\varepsilon),\widetilde{\tau})-t^*(M,\tau)|,
|t^*(M-\delta_1(\varepsilon),\widetilde{\tau})-t^*(M,\tau)|\big\}\\
&\leq&\max\big\{|t^*(M+\delta_1(\varepsilon),\widetilde{\tau})-t^*(M+\delta_1(\varepsilon),\tau)|
+|t^*(M+\delta_1(\varepsilon),\tau)-t^*(M,\tau)|,\\
&&\;\;\;\;\;\;\;\;|t^*(M-\delta_1(\varepsilon),\widetilde{\tau})-t^*(M-\delta_1(\varepsilon),\tau)|
+|t^*(M-\delta_1(\varepsilon),\tau)-t^*(M,\tau)|\big\}\\
&\leq&\varepsilon.
\end{eqnarray*}
This implies that $t^*(\cdot,\cdot)$ is continuous at $(M,\tau)$.

We then show the convergence of optimal controls. Since (\ref{impulse-3}) holds and $(M_n,\tau_n)\rightarrow (M,\tau)$,
by Lemma~\ref{2018-impulse-10}, we can assume that
\begin{equation*}
\big\{e^{\tau_n \Delta} y_0+\chi_\omega u: u\in \mathcal{U}_{M_n}\big\}\cap B_r(0)=\emptyset\;\;
\mbox{for all}\;\;n\geq 1.
\end{equation*}
This, along with (\ref{impulse-3}) and Theorem~\ref{impulse-4}, implies that
$(TP)_{M_n}^{\tau_n}$ and $(TP)_M^\tau$ have unique optimal controls $u_n^*$ and $u^*$, respectively.
Moreover,
\begin{equation}\label{proof-34}
\|u_n^*\|_{L^2(\Omega)}=M_n,\;\;\|u^*\|_{L^2(\Omega)}=M
\end{equation}
and
\begin{equation}\label{proof-35}
\begin{array}{lll}
&&y^{\tau_n}(t^*(M_n,\tau_n); y_0, u_n^*)\\
&=&e^{t^*(M_n,\tau_n)\Delta}y_0+e^{(t^*(M_n,\tau_n)-\tau_n)\Delta}(\chi_\omega u_n^*)\in B_r(0).
\end{array}
\end{equation}
We arbitrarily take a subsequence $\{u_{n_k}^*\}_{k\geq 1}$ of $\{u_n^*\}_{n\geq 1}$.
Since $\|u_{n_k}^*\|_{L^2(\Omega)}=M_{n_k}\rightarrow M$, there exists a subsequence of $\{n_k\}_{k\geq 1}$,
still denoted in the same way, and a control $\widetilde{u}\in L^2(\Omega)$ so that
\begin{equation}\label{proof-36}
u_{n_k}^*\rightarrow \widetilde{u}\;\;\mbox{weakly in}\;\;L^2(\Omega)
\;\;
\mbox{and}\;\;\|\widetilde{u}\|_{L^2(\Omega)}\leq \liminf_{k\rightarrow +\infty}\|u_{n_k}^*\|_{L^2(\Omega)}=M.
\end{equation}
By (\ref{proof-35}), the first relation of (\ref{proof-36}) and the continuity of $t^*(\cdot,\cdot)$ at $(M,\tau)$,
using the similar arguments as those in deriving  (\ref{proof-7(1)}),
we see that $y^\tau(t^*(M,\tau); y_0, \widetilde{u})\in B_r(0)$. This, along with the second relation of (\ref{proof-36}),
implies that $\widetilde{u}$ is an optimal control to the problem $(TP)_M^\tau$. Since $u^*$ is the unique optimal
control to the problem $(TP)_M^\tau$, we have that $\widetilde{u}=u^*$ and
\begin{equation}\label{proof-37}
u_{n_k}^*\rightarrow u^*\;\;\mbox{weakly in}\;\;L^2(\Omega).
\end{equation}
Noting that $\|u_{n_k}^*\|_{L^2(\Omega)}\rightarrow \|u^*\|_{L^2(\Omega)}$, by (\ref{proof-37}),
we get that $u_{n_k}^*\rightarrow u^*$ strongly in $L^2(\Omega)$. \\

In summary, we finish the proof of Theorem~\ref{impulse-5}.
\end{proof}



\end{document}